\documentclass{amsart}
\date{September 7, 2017}
\usepackage{graphicx, color, enumerate}
\usepackage{amsmath, amssymb, amscd, mathrsfs}
\title{Complete flat fronts as hypersurfaces in Euclidean space}
\author[A.~Honda]{Atsufumi Honda}
\address{%
   Department of Applied Mathematics, 
   Faculty of Engineering, Yokohama National University, 
   79-5 Tokiwadai, Hodogaya, Yokohama 240-8501, Japan
}
\email{honda-atsufumi-kp@ynu.jp}
\thanks{This work is supported by 
the Grant-in-Aid for Young Scientists (B) No.~16K17605 from JSPS}
\usepackage{amsthm}
\theoremstyle{plain}
 \newtheorem{theorem}{Theorem}[section]
 \newtheorem{introtheorem}{Theorem}

 \newtheorem{proposition}[theorem]{Proposition}
 \newtheorem{fact}[theorem]{Fact}
 \newtheorem*{fact*}{Fact}
 \newtheorem{lemma}[theorem]{Lemma}
 
 \theoremstyle{remark}
 \newtheorem{definition}[theorem]{Definition}
 
 \newtheorem*{acknowledgements}{Acknowledgements}
 
\numberwithin{equation}{section}

\newcommand{\Z}{\mathbb{Z}}
\newcommand{\R}{\mathbb{R}}

\newcommand{\inner}[2]{\left\langle{#1},{#2}\right\rangle}
\newcommand{\vect}[1]{\boldsymbol{#1}}

\usepackage{multirow}
\usepackage{array}
\newcommand\innerrule{%
   \omit{\let\strut\relax
     \hrulefill\vspace*{-\arrayrulewidth}%
   }%
}
\newcommand\vcentertext[1]{%
   \let\strut\relax
   \vbox to0pt{\vss \hbox{#1} \vss}%
}
\begin{document}
\begin{abstract}
By Hartman--Nirenberg's theorem, 
any complete flat hypersurface in Euclidean space
must be a cylinder over a plane curve.
However, if we admit some singularities,
there are many non-trivial examples. 
{\it Flat fronts} are flat hypersurfaces with admissible singularities.
Murata--Umehara gave a representation formula
for complete flat fronts with non-empty singular set
in Euclidean $3$-space,
and proved the four vertex type theorem.
In this paper, we prove that,
unlike the case of $n=2$, 
there do not exist any complete flat fronts 
with non-empty singular set in Euclidean $(n+1)$-space
$(n\geq3)$.
\end{abstract}
\maketitle

\section{Introduction}
\label{intro}
Let $\R^{n+1}$ be the Euclidean $(n+1)$-space.
By Hartman--Nirenberg's theorem\footnote{%
Massey \cite{Massey} gave an alternative proof
for $n=2$.} \cite{HN},
{\it any complete flat hypersurface in $\R^{n+1}$
must be a cylinder\footnote{%
A cylinder is a regular hypersurface
$f : \R^n \rightarrow \R^{n+1}$ given by
$
f(t,w_2,\dots,w_n):=(x(t), y(t),w_2,\dots,w_n),
$
where $t\mapsto (x(t), y(t))$ is a regular curve in $\R^2$.}
over a plane curve}.

However, in $\R^3$,
there are non-trivial flat surfaces with admissible singularities
called {\it flat fronts\/}.
Here, a {\it front} is a generalized notion 
of regular surfaces (more generally, regular hypersurfaces)
with admissible singular points.
See Section \ref{sec:prelim} for precise definitions.
Murata--Umehara gave a representation formula
for complete flat fronts with non-empty singular set,
and proved the four vertex type theorem:

\begin{fact}[\cite{MU}]\label{fact:MU}
Let $\xi : S^1 \rightarrow S^2$ 
be a regular curve without inflection points,
and $\alpha=a(t)dt$ a $1$-form on $S^1=\R/2\pi\Z$
such that 
$\int_{S^1} \xi\, \alpha=0$
holds. 
Then,
$f_{\xi,\alpha} : S^1\times \R \rightarrow \R^3$
defined by
\begin{equation}\label{eq:MU}
  f_{\xi,\alpha}(t,v) := \hat{\sigma}(t) + v\, \xi(t)\qquad
  \left( \hat{\sigma}(t):=\int_0^t a(\tau)\,\xi(\tau) d\tau \right)
\end{equation}
is a complete flat front with non-empty singular set.
Conversely, let $f : M^2 \rightarrow \R^3$ be a complete flat front
defined on a connected smooth $2$-manifold $M^2$.
If the singular set $S(f)$ of $f$ is not empty,
then $f$ is umbilic-free, co-orientable, 
$M^2$ is diffeomorphic to $S^1\times \R$,
and $f$ is given by \eqref{eq:MU}.
Moreover, if the ends of $f$ are embedded, 
$f$ has at least four singular points other than cuspidal edges.
\end{fact}

Therefore, it is natural to ask what occurs in the higher dimensional cases.
In this paper, we prove that there do not exist 
any non-trivial flat fronts in higher dimensions:

\begin{introtheorem}\label{thm:main}
If $n\geq3$, there do not exist any complete flat fronts 
with non-empty singular set in $\R^{n+1}$.
\end{introtheorem}

Combining Hartman--Nirenberg's theorem \cite{HN}, 
Murata--Umehara's theorem \cite{MU}
and Theorem \ref{thm:main}, 
we have the classification of complete flat fronts in $\R^{n+1}$.

\begin{table}[htb]
\begin{center}
\begin{tabular}{|c||c|c|}\hline
   & Singular set $=\emptyset$ & Singular set $\neq\emptyset$ \\
   \hline\hline
  $n=2$ & Cylinder & $\exists \infty$ (Murata--Umehara \cite{MU}) \\
  \innerrule &  & \innerrule \\
  $n\geq3$ & (Hartman--Nirenberg \cite{HN}) & $\not\exists$ (Theorem \ref{thm:main}) \\
  \hline
\end{tabular}
\caption{Classification of complete flat fronts in $\R^{n+1}$.}
\end{center}
\end{table}

\vspace{-4mm}

We remark that, 
although there do not exist 
any complete flat fronts in $\R^{n+1}$ $(n\geq3)$,
there are many {\it weakly complete} ones.
For example, 
we can construct a weakly complete flat front 
by a pair $(\gamma(t), a(t))$
of a complete regular curve $\gamma(t)$ in $S^n$
and a smooth function $a(t)$ on $\R$ 
(cf.~Proposition \ref{prop:representation}).
Here, we denote by $S^n$ the $n$-sphere 
of constant sectional curvature $1$.
Moreover, by a regular curve in $\R^{n+1}$,
one may construct a flat front called 
{\it tangent developable}.
(See \cite{Ishikawa} for more details
and properties of singularities of tangent developables.)

We also remark that there are several works
related to Murata--Umehara's theorem.
Naokawa \cite{Naokawa} gave
an estimation of singular points other than cuspidal edges 
on asymptotic completions of developable M\"{o}bius strips.
On the other hand, flat fronts 
can be considered as fronts with one principal curvature zero.
In a previous paper \cite{Honda2}, 
the author gave a classification of weakly complete 
fronts with one principal curvature non-zero constant.

With respect to the case of non-flat ambient spaces,
it is known that
flat fronts in $\R^{n+1}$ is identified with 
fronts of constant sectional curvature $1$ 
(CSC-$1$ fronts) in $S^{n+1}$
via the central projection of a hemisphere to a tangent space. 
Therefore, the local nature of flat fronts in $\R^{n+1}$
is the same as that of CSC-$1$ fronts in $S^{n+1}$.
However, they may display different global properties.
In \cite{Honda3}, the author gave a classification of 
complete CSC-$1$ fronts, which is a generalization of 
O'Neill-Stiel's theorem \cite{OS}.
In particular, in the case of $n\geq3$, 
there exist many non-trivial complete CSC-$1$ fronts in $S^{n+1}$,
although there do not exist any complete flat fronts other than cylinders in $\R^{n+1}$.
(See also \cite{Honda1} for the case of negative sectional curvature.)

This paper is organized as follows.
In Section \ref{sec:prelim}, 
we shall review the definition and fundamental properties 
of flat fronts.
Using them, we shall prove Theorem \ref{thm:main} in Section \ref{sec:proof}.

\section{Preliminaries}
\label{sec:prelim}

We denote by $\R^{n+1}$ the Euclidean $(n+1)$-space,
and $S^n$ the unit sphere 
$S^n:=\{ \vect{x}\in \R^{n+1}\,;\, \vect{x}\cdot \vect{x}=1 \},$
where the dot `$\cdot$' is the canonical inner product on $\R^{n+1}$.
Let $M^n$ be a connected smooth $n$-manifold
and 
$$f:M^n\rightarrow \R^{n+1}$$ a smooth map.
A point $p\in M^n$ is called a \emph{singular point} 
if $f$ is not an immersion at $p$.
Otherwise, we say $p$ a \emph{regular point}.
Denote by $S(f)$ $(\subset M^n)$ 
be the set of singular points.
If $S(f)$ is empty, we call $f$ a (regular) hypersurface.

A smooth map $f:M^n\rightarrow \R^{n+1}$
is called a \emph{frontal\/}, if for each point $p\in M^n$,
there exist a neighborhood $U$ of $p$ and 
a smooth map $\nu : U \rightarrow S^{n}$
such that 
$df_q(\vect{v})\cdot\nu(q)=0$ 
holds for each $q\in U$ and $\vect{v} \in T_qM^n$.
Such a $\nu$ is called the unit normal vector field or the Gauss map of $f$.
If $\nu$ can be defined throughout $M^n$,
$f$ is called {\it co-orientable}.
On the other hand, we say {\it orientable} if $M^n$ is orientable.
If 
$$
  (L:=)~(f,\nu) : U\rightarrow \R^{n+1}\times S^n
$$
 gives an immersion,
$f$ is called a \emph{wave front\/} (or a \emph{front\/}, for short).
The map $L=(f,\nu)$ is called the Legendrian lift of $f$.

\subsection{Completeness, Weak completeness, Umbilic points}
The first fundamental form (i.e., the induced metric)
is given by $ds^2:=df\cdot df$.
For a front $f : M^n \rightarrow \R^{n+1}$
with a (possibly locally defined) unit normal vector field $\nu$,
$$
ds^2_\#:=ds^2+d\nu\cdot d\nu
$$
gives a positive definite Riemannian metric,
and called the lift metric.
If the lift metric $ds^2_\#$ is complete, 
$f$ is called {\it weakly complete}.
On the other hand,
$f$ is called {\it complete},
if there exists a symmetric covariant $(0,2)$-tensor $T$ on $M^n$
with compact support such that
$ds^2+T$ gives a complete metric on $M^n$.
In this case, the singular set $S(f)$ must be compact.
As noted in \cite{MU}, if $S(f)$ is empty, 
then $f : M^n\rightarrow \R^{n+1}$ is complete as a front if and only if 
$f$ is complete as a regular hypersurface
(i.e., $(M^n, ds^2)$ is a complete Riemannian manifold).

\begin{fact}[{\cite[Lemma 4.1]{MU}}]
\label{fact:MU_complete}
A complete front is weakly complete.
\end{fact}

A point $p\in M^n$ is called an {\it umbilic point},
if there exist real numbers $\delta_1$, $\delta_2$ 
such that 
$\delta_1(df)_p=\delta_2(d\nu)_p,$
$(\delta_1,\delta_2)\neq(0,0)$
hold.
For a positive number $\delta>0$, set
\begin{equation}\label{eq:parallel}
  f^{\delta}:= f+\delta \nu,\qquad
  \nu^{\delta}:=\nu.
\end{equation}
Then we can check that 
$f^{\delta}$ is a front 
and $\nu^{\delta}$ gives a unit normal along $f^{\delta}$.
Such an $f^{\delta}$ is called the {\it parallel front} of $f$.
Umbilic points are common in its parallel family.

\begin{fact}[{\cite[Lemma 2.7]{Honda3}}]
\label{lem:singular-umb}
Let $p\in M^n$ be a singular point
of a front $f$.
Then, $p$ is an umbilic point if and only if ${\rm rank}(df)_p=0$ holds.
In this case, we have ${\rm rank}(d\nu)_p=n$.
\end{fact}

\subsection{Flat fronts}

In \cite{SUY0, SUY1}, Saji--Umehara--Yamada 
introduced {\it coherent tangent bundles\/}%
\footnote{It is known that, in general, coherent tangent bundles can be constructed
only from positive semi-definite metrics called {\it Kossowski metrics}
(cf.~\cite{HHNSUY, SUY3}).},
which is a generalized notion of Riemannian manifolds.

Let $\mathcal{E}$ be a vector bundle of rank $n$
over a smooth $n$-manifold $M^n$.
We equip a fiber metric $\inner{~}{~}$ on $\mathcal{E}$
and a metric connection $D$ on $(\mathcal{E},\inner{~}{~})$.
Let $\varphi : TM^n \rightarrow \mathcal{E}$
be a bundle homomorphism such that
\begin{equation}\label{eq:Codazzi}
  D_{X}\varphi(Y)-D_{Y}\varphi(X)-\varphi([X,Y])=0
\end{equation}
holds for arbitrary smooth vector fields $X$, $Y$ on $M^n$.
Then 
$\mathcal{E}=(\mathcal{E},\,\inner{~}{~},\,D,\,\varphi)$
is called a {\it coherent tangent bundle} over $M^n$.

We shall review the coherent tangent bundles 
induced from frontals
(cf.~ \cite[Example 2.4]{SUY1}).
For a frontal $f : M^n\rightarrow \R^{n+1}$,
set 
$\mathcal{E}_f$, $\inner{~}{~}_f$, $D_f$ and $\varphi_f$,
respectively, as follows:
\begin{itemize}
\item
 $\mathcal{E}_f$ is the subbundle of the pull-back bundle $f^*\R^{n+1}$ 
 perpendicular to $\nu$,
\item
 $\inner{~}{~}_f$ is the metric on $\mathcal{E}_f$ induced from 
 the canonical metric on $\R^{n+1}$,
\item
 $D_f$ is the tangential part 
 of the Levi-Civita connection on $\R^{n+1}$,
\item
 $\varphi_f : TM^n \rightarrow \mathcal{E}_f$ 
 defined as $\varphi_f(X):=df(X)$.
\end{itemize}
Then, $\mathcal{E}_f=(\mathcal{E}_f,\inner{~}{~}_f,D_f,\varphi_f)$
is a coherent tangent bundle,
which we call the \emph{induced coherent tangent bundle}.

\begin{definition}[\cite{Honda3}]
A coherent tangent bundle 
is said to be {\it flat}
if $$R^D(X,Y)\xi=0$$ holds
for all smooth vector fields $X$, $Y$ on $M^n$
and each smooth section $\xi$ of $\mathcal{E}$, 
where $R^D$ is the curvature tensor of the connection $D$ given by
\[
  R^D(X,Y)\xi:=D_XD_Y\xi-D_YD_X\xi-D_{[X,Y]}\xi.
\]
A frontal $f$ is called {\it flat},
if the induced coherent tangent bundle $\mathcal{E}_f$ 
is flat.
\end{definition}

In \cite{Honda3},
the following characterization of flatness was proved 
by using the Gauss equation for frontals
given by Saji--Umehara--Yamada \cite[Proposition 2.4]{SUY2}.

\begin{fact}[{\cite[Lemma 3.3]{Honda3}}]
\label{lem:degenerate-nu}
Let $f : M^n\rightarrow \R^{n+1}$ 
be a frontal with a unit normal vector field $\nu$.
Then $f$ is flat if and only if 
\begin{equation}\label{eq:d-nu}
   {\rm rank}(d\nu)\leq1
\end{equation}
holds on $M^n$.
\end{fact}

We remark that
Murata--Umehara \cite{MU} defined 
the flatness for frontals in $\R^3$
by the condition \eqref{eq:d-nu}.
Therefore, our definition of flatness is compatible to 
that given by Murata--Umehara.

\section{Proof of Theorem \ref{thm:main}}
\label{sec:proof}

Denote by $\mathcal{U}_f$ the set of umbilic points.

\begin{lemma}\label{lem:sp_coord}
Let $f : M^n \rightarrow \R^{n+1}$ be a 
non-totally-umbilic flat front.
For each non-umbilic point
$q\in M^n\setminus\mathcal{U}_f$,
there exist a local coordinate neighborhood
$(U\,;\,u_1,\dots,u_n)$ of $q$
and a smooth function $\rho=\rho(u_1,\dots,u_n)$ on $U$ such that
\begin{equation}\label{eq:sp_coord}
  -\rho \nu_{u_1}=f_{u_1},\quad
  \nu_{u_j}=0,\quad
  \nu_{u_1}\cdot f_{u_j}=0
\end{equation}
hold for each $j=2,\dots, n$, 
and $\{\nu_{u_1}, f_{u_2},\dots, f_{u_n}\}$ 
is a frame on $U$.
For each $u_1$, set the slice $U_{u_1}$ of $U$ as
$
U_{u_1}:=\{\vect{u} \in \R^{n-1} \,;\,
 (u_1,\vect{u})\in U\}.
$
Then, the restriction $f|_{U_{u_1}} : U_{u_1} \rightarrow \R^{n+1}$ is 
a totally geodesic embedding for each $u_1$.
\end{lemma}

\begin{proof}
Since $f$ is flat, 
Lemma \ref{lem:degenerate-nu} implies that
there exists a local coordinate system
$(V\,;\,v_1,\dots,v_n)$ around 
$q\in M^n\setminus\mathcal{U}_f$ such that 
$\nu_{v_j}=0$
for each $j=2,\dots, n$.
If $f_{v_j}=0$ for some $j=2,\dots, n$, 
then $L_{v_j}=(f_{v_j},\, \nu_{v_j})=(0,0)$ holds,
which contradicts the condition that $f$ is a front. 
Therefore, 
$f_{v_j}\neq0$ for each $j=2,\dots, n$.
Hence $p\in V$ is a singular point 
if and only if $f_{v_1}(p)=0$, and then 
$\{\nu_{v_1},\, f_{v_2},\dots,\, f_{v_n}\}$ is linearly independent.
In this case, for each $\delta \neq0$,
the parallel front $f^{\delta}:=f+\delta\,\nu$  is 
a flat immersion around $p$
(cf.\ \eqref{eq:parallel}).
Since $f$ is umbilic-free, so is $f^{\delta}$ for each $\delta \neq0$.
Let $(U\,;\,u_1,\dots,u_n)$ be a curvature line coordinate system of 
$f^{\delta}$ around $q\in M^n\setminus\mathcal{U}_f$.
That is, 
\begin{equation}\label{eq:sp_coord_parallel}
  (\nu^{\delta})_{u_j}=0,
  \quad
  (f^{\delta})_{u_1}\cdot(f^{\delta})_{u_j}=0,
  \quad
  -(\nu^{\delta})_{u_1} = \alpha(f^{\delta})_{u_1}
\end{equation}
hold, where $j=2,\dots, n$,
and $\alpha=\alpha(u_1,\dots,u_n)$
is a smooth function on $U$.
In this case, 
the principal curvatures 
$\lambda_1^\delta,\dots,\lambda_n^\delta$
of $f^{\delta}$ are given by
$\lambda_1^\delta = \alpha$, $\lambda_j^\delta=0$
$(j=2,\dots,n)$.
Since $f^{\delta}$ is umbilic-free,
$\alpha\neq 0$ on $U$.
Substituting \eqref{eq:parallel}
into equation \eqref{eq:sp_coord_parallel}, 
we may conclude that 
\eqref{eq:sp_coord} holds
with $\rho:=(1+\delta\alpha)/\alpha.$
With respect to the third assertion,
$\vect{n}:=\nu_{u_1}/|\nu_{u_1}|$ gives a unit normal vector field of $f|_{U_{u_1}}$.
Set $\psi:=1/|\nu_{u_1}|$.
Then, for each $j=2,\dots,n$, 
\begin{align*}
  \vect{n}_{u_j}
  = \psi_{u_j}\nu_{u_1}+\psi\nu_{u_1u_j}
  =\psi^{-1} \psi_{u_j} \vect{n}
\end{align*}
and $\vect{n}\cdot \vect{n}_{u_j}=0$ yield
$\vect{n}_{u_j}=0$ on $U_{u_1}$.
Together with $\nu_{u_j}=0$ $(j=2,\dots,n)$ on $U_{u_1}$,
we have the conclusion.
\end{proof}

By Lemma \ref{lem:sp_coord},
since the image of $f|_{U_{u_1}}$ is included in 
a $(n-1)$-dimensional affine subspace $A^{n-1}_{u_1}$ of $\R^{n+1}$
for each $u_1$, 
by a coordinate change of $(u_2,\dots,u_n)$, 
we may take a new coordinate system 
$(U'; u_1, w_2,\dots,w_n)$
such that
$(w_2,\dots,w_n)$ is the canonical Euclidean coordinate system of
$A^{n-1}_{u_1}$ for each ${u_1}$. 
Namely, $f_{w_j}\cdot f_{w_k}=\delta_{jk}$ holds for $j,k=2,\dots,n$.

Setting 
$\sigma(u_1):=f(u_1,0,\dots,0)$ 
and 
$\vect{e}_j(u_1):=f_{w_j}(u_1,0,\dots,0)$ 
for $j=2,\dots,n$, 
we have 
\begin{equation*}
  f(u_1,w_2,\dots,w_n)
  =\sigma(u_1) + w_2 \vect{e}_2(u_1) +\cdots + w_n \vect{e}_n(u_1).
\end{equation*}
Since $f$ has no umbilic point on $U$,
the Gauss map $\nu$ depends only on $u_1$ and $\nu_{u_1}\neq0$ holds.
Therefore, 
$$
\gamma(u_1):=\nu(u_1,0,\dots,0)
$$
is a regular curve in $S^n$.
By a coordinate change of $u_1$, 
we may take a new coordinate system 
$(W; t, w_2,\dots,w_n)$ such that 
the spherical regular curve
$t\mapsto \gamma(t)$ is parametrized by arc-length.
Thus, we have
\begin{equation}\label{eq:representation}
  f(t,w_2,\dots,w_n)=
  \sigma(t) + w_2 \vect{e}_2(t) +\cdots + w_n \vect{e}_n(t).
\end{equation}

Denote by 
$\vect{e}(t):=\gamma'(t)$ the unit tangent vector of $\gamma(t)$.
Since $f_{w_j}=\vect{e}_j$ for each $j=2,\dots, n$ and
$\gamma(t)$ is the Gauss map of $f$,
we have 
\begin{equation}\label{eq:frame-1}
  \gamma(t)\cdot \vect{e}_j(t)=0\qquad (j=2,\dots, n).
\end{equation}
In addition, the third equation of 
\eqref{eq:sp_coord} yields
\begin{equation}\label{eq:frame-2}
  \gamma'(t)\cdot \vect{e}_j(t)=0\qquad (j=2,\dots, n).
\end{equation}
Therefore,
$\{\vect{e}_j(t)\}_{j=2,\dots,n}$
is an orthonormal frame of the normal bundle $\left(\gamma'(t)\right)^\perp$ 
along the spherical regular curve $\gamma(t)$.
Moreover, equations \eqref{eq:frame-1} and \eqref{eq:frame-2} yield
\begin{equation}\label{eq:frame-3}
  \gamma(t)\cdot \vect{e}_j'(t)=0\qquad (j=2,\dots, n).
\end{equation}
Hence, by \eqref{eq:representation},
$f_{t}\cdot \gamma=0$ implies
$\sigma'(t)\cdot \gamma(t)=0 $.
Therefore, there exist smooth functions
$a_j=a_j(t)$ $(j=1,\dots, n)$ such that
\begin{equation}\label{eq:sigma-prime}
  \sigma'(t)=a_1(t) \vect{e}(t)
                         + a_2(t) \vect{e}_2(t) +\cdots 
                         + a_n(t) \vect{e}_n(t).
\end{equation}
Thus, we have the following:

\begin{lemma}\label{lem:representation}
Let $f : M^n \rightarrow \R^{n+1}$ be a 
non-totally-umbilic flat front.
For each non-umbilic point
$q\in M^n\setminus\mathcal{U}_f$,
there exist a local coordinate neighborhood
$(W\,;\, t, w_2,\dots,w_n)$ of $q$,
a regular curve $\gamma(t)$ in $S^n$,
a orthonormal frame $\{\vect{e}_2(t), \dots , \vect{e}_n(t)\}$
of the normal bundle $(\gamma')^{\perp}$ along $\gamma(t)$
and smooth functions $\{a_j(t)\}_{j=1,\dots,n}$
such that $f$ is given by \eqref{eq:representation} on $W$,
where
\begin{equation}\label{eq:sigma}
  \sigma(t):=\int_0^{t} \eta(\tau)\,d\tau\qquad
  \left( \eta(\tau):= a_1(\tau)\vect{e}(\tau)+\sum_{j=2}^n a_j(\tau)\vect{e}_j(\tau) \right)
\end{equation}
and $\vect{e}(t):=\gamma'(t)$.
\end{lemma}

Finally, we shall reduce the numbers of functions.
For a unit speed regular curve 
$\gamma=\gamma(t) : I\rightarrow S^n$
defined on an open interval $I$,
set $\vect{e}(t):=\gamma'(t)$.
Then, there exist an orthonormal frame $\{\vect{e}_j(t)\}_{j=2,\dots,n}$
of the normal bundle along $\gamma$ 
and smooth functions $\mu_j(t)$ $(j=2,\dots,n)$
such that 
\begin{equation}\label{eq:Bishop}
\vect{e}_j'(t)=-\mu_j(t)\vect{e}(t)
\qquad(j=2,\dots,n).
\end{equation}
Such a frame $\{\vect{e}_j(t)\}_{j=2,\dots,n}$
is called the {\it Bishop frame} (cf.~\cite{Bishop}).

Let $f=f(t,w_2,\dots,w_n)$ 
be the flat front given by \eqref{eq:representation}
with the Bishop frame $\{\vect{e}_j(t)\}_{j=2,\dots,n}$.
Set
$$
  \rho(t,w_2,\dots,w_n):=a_1(t)-\sum_{j=2}^nw_j \mu_j(t).
$$
Since
$f_{w_j}=\vect{e}_j(t)$ for $j=1,\dots,n$,
\begin{equation*}
  f_t=\rho(t,w_2,\dots,w_n)\,\vect{e}(t)
          +a_2(t)\vect{e}_2(t)+\cdots+a_n(t)\vect{e}_n(t),
\end{equation*}
and $d\nu\cdot d\nu=dt^2$,
the lift metric $ds^2_\#=df\cdot df+d\nu\cdot d\nu$ is given by
\begin{equation*}
  ds^2_\#=\left( 1+ \rho^2 +\sum_{j=2}^n (a_j(t))^2 \right)dt^2
            +\sum_{j=2}^n \left(2a_j(t)dw_jdt+(dw_j)^2\right).
\end{equation*}
By a straightforward calculation, 
it can be checked that 
each $w_j$-curve $(j=2,\dots,n)$
gives a geodesic of the lift metric $ds^2_\#$.
Thus, if $f$ is weakly complete,
every $w_j$-curve $(j=2,\dots,n)$
can be defined on the whole real line $\R$.
By a coordinate change
$$
(t, w_2,\dots,w_n)\mapsto (t,w_2+b_2(t), \dots ,w_n+b_n(t) )
\quad\left( b_j(t):=-\int_0^t a_j (\tau)d\tau.\right)
$$
we have
\begin{align*}
  f(t,w_2+b_2(t), \dots ,w_n+b_n(t) )
  &= \sigma(t) + \sum_{j=2}^n(w_j+b_j(t))\vect{e}_j(t)\\
  &= \hat\sigma(t) + \sum_{j=2}^n w_j\vect{e}_j(t),
\end{align*}
where we set 
$
  \hat\sigma(t):=\sigma(t) + b_2(t)\vect{e}_2(t)+ \cdots + b_n(t)\vect{e}_n(t).
$
By \eqref{eq:sigma},
we have
$
  \hat\sigma'(t) = a(t) \vect{e}(t),
$
where 
$
  a(t):=a_1(t) - b_2(t)\mu_2(t)- \cdots - b_n(t)\mu_n(t).
$
Therefore, we have the following:

\begin{proposition}\label{prop:representation}
Let $f : M^n \rightarrow \R^{n+1}$ be a 
weakly complete flat front which is not totally-umbilic.
Around each non-umbilic point,
there exist an interval $I$, 
a local coordinate system
$(I\times \R^{n-1} \,;\, t, w_2,\dots,w_n)$,
a regular curve $\gamma : I\rightarrow S^n$
parametrized by arc-length,
an orthonormal frame $\{\vect{e}_2, \dots , \vect{e}_n\}$
of the normal bundle $(\gamma')^{\perp}$ along $\gamma$
and a smooth function $a(t)$ on $I$
such that $f$ is given by 
\begin{equation}\label{eq:representation-hat}
  f(t,w_2,\dots,w_n) = \hat\sigma(t) + \sum_{j=2}^n w_j\vect{e}_j(t)\qquad
  \left( \hat\sigma(t):=\int_0^t  a(\tau) \gamma'(\tau)  d\tau \right)
\end{equation}
on $I\times \R$.
Conversely,
for a given unit speed regular curve 
$\gamma : I\rightarrow S^n$
defined on an interval $I$,
an orthonormal frame $\{\vect{e}_2, \dots , \vect{e}_n\}$
of the normal bundle $(\gamma')^{\perp}$ along $\gamma$
and a smooth function $a(t)$ on $I$,
$f: I\times \R^{n-1} \rightarrow \R^{n+1}$
defined as \eqref{eq:representation-hat}
is an umbilic-free flat front.
\end{proposition}

\begin{proof}[Proof of Theorem \ref{thm:main}]
We shall give a proof by contradiction.
Let $f : M^n \rightarrow \R^{n+1}$ a complete flat front $(n\geq3)$.
Assume that the singular set $S(f)$ is not empty.

Take a singular point $q\in S(f)$.
By Facts \ref{lem:singular-umb} and \ref{lem:degenerate-nu},
$q$ is not an umbilic point.
Since $f$ is complete,
it is weakly complete and, 
by Proposition \ref{prop:representation},
we have that $f$ is given by \eqref{eq:representation-hat}
on $U:=I\times \R^{n-1}$.
Without loss of generality, 
$\{\vect{e}_2, \dots , \vect{e}_n\}$
is the Bishop frame such that 
$
\vect{e}_j'(t)=-\mu_j(t)\vect{e}(t)
$
holds for each $j=2,\dots,n$ (cf.\ \eqref{eq:Bishop}).
We remark that the curvature function 
$\kappa_{\gamma}(t)$ of $\gamma(t)$
is given by
$$
  \kappa_{\gamma}(t) = \sqrt{(\mu_2(t))^2+\cdots +(\mu_n(t))^2}.
$$

We shall prove that the singular set $S(f)$ is not compact.
Differentiating \eqref{eq:representation-hat},
we have 
$f_t = \hat\rho(t,w_2,\dots,w_n)\vect{e}(t),$
$f_{w_j}=\vect{e}_j(t)$
for $j=1,\dots,n$, where
$$
  \hat\rho(t,w_2,\dots,w_n) := a(t)-\sum_{j=2}^nw_j \mu_j(t).
$$
Since
$
  f_t\wedge f_{w_2} \wedge \cdots \wedge f_{w_n}
  = \hat\rho(t,w_2,\dots,w_n) 
        \vect{e}(t)\wedge \vect{e}_2(t) \wedge \cdots \wedge \vect{e}_n(t)
$,
we have
$
  S(f)\cap U = \{ p \in U\,;\,  \hat\rho(p)=0 \}.
$
Let $S_1$, $S_2$ be the subsets of $S(f)\cap U$ defined by
\begin{align*}
  S_1 &:= \{ (t,w_2,\dots,w_n) \in U \,;\, 
          a(t)=\kappa_\gamma(t)=0 \},\\
  S_2 &:= \{ (t,w_2,\dots,w_n) \in U \,;\, 
         \hat\rho(t,w_2,\dots,w_n)=0, \kappa_\gamma(t)\neq0 \},
\end{align*}
respectively.
Then, we have $S(f)\cap U=S_1\cup S_2$.

Since $\nu(t,w_2,\dots,w_n)=\gamma(t)$ 
gives a unit normal vector field along $f$,
the lift metric $ds^2_\#$ is given by
$$
  ds^2_\#
  = \left(1 + \hat\rho^2\right) dt^2 + \sum_{j=2}^n dw_j^2
$$
on $U$.

If $q=(t^o,w_2^o,\dots,w_n^o)\in S_1$, 
$a(t^o)=\kappa_\gamma(t^o)=0$ holds.
In this case, we have
$(t^o,w_2,\dots,w_n)\in S_1$ for any 
$w_j\in \R$ $(j=2,\dots,n)$.
In particular, 
$c_1 : \R \rightarrow S_1 \,(\subset M^n)$
given by
$$
c_1(x):= (t^o,x,0,\dots,0) 
$$
is a geodesic with respect to the lift metric $ds^2_\#$
such that $\hat{c}_1:=f\circ c_1$ is a straight line in $\R^{n+1}$,
and hence $ S(f)\, (\supset S_1)$ cannot be compact.

If $q=(t^o,w_2^o,\dots,w_n^o)\in S_2$, 
we have $\kappa_\gamma(t^o)\neq0$.
Without loss of generality, 
we may assume that $\mu_n(t^o)\neq0$.
Then, there exists $\varepsilon >0$ such that 
$\mu_n(t)\neq0$
for each $t\in I(t^o,\varepsilon):=(t^o-\varepsilon,t^o+\varepsilon)$.
Thus,
\begin{equation*}
  S_2[t^o] := \left\{ (t,w_2,\dots,w_n) 
                    \in I(t^o,\varepsilon)\times \R^{n-1} \,;\, 
                       \vphantom{\frac1{6}} 
                     w_n = \frac{a(t)}{\mu_n(t)}
                           -\sum_{j=2}^{n-1}\hat{\mu}_j(t) w_j 
                    \vphantom{\frac1{6}} \right\}
\end{equation*}
is a subset of $S_2$,
where $\hat{\mu}_j(t) := \mu_j(t)/\mu_n(t)$ for $j=2,\dots, (n-1)$.
Set a positive number $k^o$ as $k^o:= \sqrt{1+(\hat{\mu}_2(t^o))^2}$.
Since $c_2 : \R \rightarrow S_2[t^o]$
given by
$$
c_2(x):= 
\left( t^o,\frac1{k^o}x,0,\dots,0,
\frac{a(t^o)}{\mu_n(t^o)}-\frac{\hat{\mu}_2(t^o)}{k^o} x  \right) 
$$
is a geodesic with respect to the lift metric $ds^2_\#$
such that $\hat{c}_2:=f\circ c_2$ is a straight line in $\R^{n+1}$,
and hence $S(f)\, (\supset S_2[t^o])$ cannot be compact.

By the completeness of $f$, 
the singular set $S(f)$ must be compact, which is a contradiction.
Hence, we have that $S(f)$ must be empty,
and then $f$ is a complete flat regular hypersurface.
\end{proof}

\begin{acknowledgements}
The author would like to thank 
Professors Masaaki Umehara and Kotaro Yamada
for their valuable comments.
\end{acknowledgements}

\end{document}